\documentclass[11pt, reqno]{amsart}
\usepackage{amsmath}%
\usepackage{amssymb}%
\usepackage[english]{babel}
\usepackage[utf8]{inputenc}
\usepackage{hyperref}
\usepackage{hyphenat}
\usepackage{amsfonts}%
\usepackage{setspace}
\usepackage{enumerate}
\usepackage{graphicx}
\usepackage{mathabx}
\usepackage[T1]{fontenc}
\usepackage{appendix}
\usepackage{xcolor}
    \numberwithin{equation}{section}

\newcommand{\diag}{{diag}}

\newcommand{\F}{\mathbb{F}}
\newcommand{\R}{\mathbb{R}}
\newcommand{\bP}{\mathbb{P}}
\newcommand{\Hq}{\mathbb{H}}

\newcommand{\N}{\mathbb{N}}
\newcommand{\C}{\mathbb{C}}

\newcommand{\bR}{\mathbb{R}}
\newcommand{\bS}{\mathbb{S}}
\newcommand{\B}{\mathcal{B}}

\newcommand{\K}{\mathcal{K}}
\newcommand{\Hi}{\mathcal{H}}
\newcommand{\Hc}{\mathcal{H}_{\C}}

\newcommand{\U}{\mathcal{U}}
\newcommand{\Conv}{{\rm conv}}
\newcommand{\Ran}{{\rm ran}}

\newcommand{\vmin}{\underline{v}}
\newcommand{\cmin}{\underline{c}}

\newcommand{\Span}{\text{span}}

\newcommand\restr[2]{{
  \left.\kern-\nulldelimiterspace 
  #1 
  \right|_{#2} 
  }}

\newtheorem{theorem}{Theorem}[section]

 \newtheorem{corollary}[theorem]{Corollary}
 
 \newtheorem{lemma}[theorem]{Lemma}
 \newtheorem{proposition}[theorem]{Proposition}
 \theoremstyle{definition}
 \newtheorem{definition}[theorem]{Definition}

 \newtheorem{example}[theorem]{Example}
 \numberwithin{equation}{section}

\makeatletter
\newtheorem*{rep@theorem}{\rep@title}
\newcommand{\newreptheorem}[2]{
\newenvironment{rep#1}[1]{%
\def\rep@title{#2 \ref{##1}}%
\begin{rep@theorem}}%
{\end{rep@theorem}}}
\makeatother

\newreptheorem{corollary}{Corollary}

\begin{document}

\thanks{The second author was partially supported by FCT through project UID/MAT/04459/2020 and the third author was partially supported by FCT through CMA-UBI, project: UIDB/00212/2020.}

\title{S-spectrum and numerical range of a quaternionic operator}

\author[L. Carvalho]{Lu\'{\i}s Carvalho}
\address{Lu\'{\i}s Carvalho, ISCTE - Lisbon University Institute\\    Av. das For\c{c}as Armadas\\     1649-026, Lisbon\\   Portugal}
\email{luis.carvalho@iscte-iul.pt}
\author[Cristina Diogo]{Cristina Diogo}
\address{Cristina Diogo, ISCTE - Lisbon University Institute\\    Av. das For\c{c}as Armadas\\     1649-026, Lisbon\\   Portugal\\ and \\ Center for Mathematical Analysis, Geometry,
and Dynamical Systems\\ Mathematics Department,\\
Instituto Superior T\'ecnico, Universidade de Lisboa\\  Av. Rovisco Pais, 1049-001 Lisboa,  Portugal
}
\email{cristina.diogo@iscte-iul.pt}
\author[S. Mendes]{S\'{e}rgio Mendes}
\address{S\'{e}rgio Mendes, ISCTE - Lisbon University Institute\\    Av. das For\c{c}as Armadas\\     1649-026, Lisbon\\   Portugal\\ and Centro de Matem\'{a}tica e Aplica\c{c}\~{o}es \\ Universidade da Beira Interior \\ Rua Marqu\^{e}s d'\'{A}vila e Bolama \\ 6201-001, Covilh\~{a}}
\email{sergio.mendes@iscte-iul.pt}

\date{\today}

\maketitle

\begin{abstract}
We study the numerical range of bounded linear operators on quaternionic Hilbert spaces and its relation with the S-spectrum. The class of complex operators on quaternionic Hilbert spaces is introduced and the upper bild of normal complex operators is completely characterized in this setting.
\end{abstract}
\maketitle
\maketitle
\section*{Introduction}
Let $\F$ denote the fields of real numbers $\R$, complex numbers $\C$ or the skew field of Hamilton's quaternions $\Hq$. Let $T$ be a bounded linear operator on a Hilbert space $\mathcal{H}$ over $\F$, with inner product $\langle .,.\rangle$. The numerical range of $T$ is the image of the unit circle $\bS_{\Hi}\subset\Hi$ under the quadratic form $f(x)=\langle Tx,x \rangle$ from $\Hi$ to $\F$. In other words, it is the subset of $\F$
\[W_{\F}(T) = \{\langle Tx, x\rangle: \|x\|=1\},\]
where $\|\,x\,\|=\langle x, x\rangle^{\frac{1}{2}}$ is the induced norm. The geometric structure of $W_{\F}(T)$ depends on the ground field $\F$. Namely,
when $\mathcal{H}$ is a real or complex Hilbert space, $W_{\F}(T)$ is a convex set, as stated by the Toeplitz-Hausdorff Theorem (see \cite{GR}). However, when $\F=\Hq$, convexity of $W_{\Hq}$ may fail (even for one-dimensional $\Hi$). Throughout the paper we let $W(T):=W_{\Hq}(T)$ denote the quaternionic numerical range of $T$.

The numerical range of operators over quaternionic finite dimensional Hilbert spaces was introduced by Kippenhahn \cite{Ki} in 1951. For every quaternion $q\in W(T)$ the similarity class $[q]$ is contained in $W(T)$ and for that reason, we may choose for representatives of that class the complex numbers $s$ or $s^*$ in $[q]\cap \C$. This observation led Kippenhahn to introduce the bild of $T$, $B(T)=W(T)\cap\C$, as a complex set that reflects many of the properties of $W(T)$. An example is precisely convexity: $W(T)$ is a convex subset of $\Hq$ if, and only if, $B(T)$ is convex subset of $\C$. However, the upper-bild $B^+(T)=W(T)\cap\C^+$, is always a convex subset (here, $\C^+$ denotes the closed upper-half complex plane).

There is a substantial body of work on the subject of numerical range for operators on complex Hilbert spaces, both in finite and infinite dimension. Much different is the case of quaternionic linear operators. While earlier studies on the quaternionic numerical range have been carried out in finite dimension by the work of Kippenhahn \cite{Ki}, Au-Yeung \cite{Ye1,Ye2}, So and Thompson \cite{ST} and, more recently, Kumar \cite{K} and the authors \cite{CDM1,CDM2,CDM3,CDM4,CDM5}, there is a lack of results in infinite dimensional, quaternionic Hilbert spaces. A possible explanation for this
is the nonexistence, until recently, of a suitable notion of spectrum in the quaternionic setting. Although the spectrum is a fundamental notion in physics, where a quaternionic model of quantum mechanics exists since 1936, thanks to the work of Birkhoff and Von Neumann \cite{BvN}, only in 2006 the appropriate notion of spectrum was found. Colombo and Sabadini proposed the notion of  $S$-spectrum (see page 6 of \cite{CGK} for an account on the history of S-spectrum) which lead to spectral theorems for quaternionic normal operators, see \cite{ACK}.
It is worth mentioning that there are several books devoted to functional calculus and to the spectral theory on the S-spectrum, see for instance \cite{CSS} and \cite{CG}.

In this paper we study the numerical range for bounded linear operators on quaternionic Hilbert spaces and its relation with the $S$-spectrum, extending to the quaternionic setting some results of complex Hilbert spaces.
In addition, we generalize results from \cite{CDM4} and \cite{CDM5}
on the shape of the bild and the upper-bild to infinite dimensional quaternionic Hilbert spaces.

The organization of the paper is as follows. In Section 1 we introduce concepts and the theory that provides a background for our work. In Section 2 we show that the $S$-spectrum and the closure of the numerical range are invariant under
approximate unitary equivalence (Proposition \ref{S-spc and nr closed under a.u.}). The class of complex operators on a quaternionic Hilbert space is introduced and it is proved that the $S$-spectrum of a complex operator is given by the similarity classes of its $\C$-spectrum, see Proposition \ref{Prop Sspectrum complex op}. In Theorem \ref{Theo_Sspecturm_NR} it is shown that the $S$-spectrum is contained in the closure of quaternionic numerical range.
Section 3 is devoted to the study of the numerical range of complex operators on quaternionic Hilbert spaces. After describing the bild of $T$ (Proposition \ref{prop bild complex op}), a characterization of the upper-bild as the convex hull of the sets $\overline{{W}(T)\cap\C^+}$, $\overline{{W}(T^*)\cap\C^+}$ and two real numbers $\underline{v}$ and $\overline{v}$ is obtained, see Theorem \ref{theo upper bild}. In Section 4 we focus on the case of normal operators and we show that every quaternionic normal operator is approximately unitarily equivalent to a complex diagonal operator (see Proposition \ref{prop normal app diagonal}). We are then able to prove that the closure of the upper bild of $T$ is the convex hull of $\sigma_S(T)\cap\C^+$ and the above mentioned real values $\underline{v}$ and $\overline{v}$ (see Theorem \ref{theo_upperbild_spectrum}). Finally, the theory developed so far is applied to characterize the real numbers $\underline{v}$ and $\overline{v}$, from pairs of eigenvalues of $T$ (Theorem \ref{theo_vmin_vmax}).

\section{Preliminaries}

The division ring of real quaternions $\Hq$ is an algebra over $\R$ with basis $\{1, i, j, k\}$ and  product given by $i^2=j^2=k^2=ijk=-1$. The pure quaternions are denoted by $\bP=\mathrm{span}_{\R}\,\{i,j,k\}$. The real and imaginary parts of a quaternion $q=a_0+a_1i+a_2j+a_3k\in\Hq$ are denoted by $Re(q)=a_0$ and $Im(q)=a_1i+a_2j+a_3k\in\bP$, respectively. The conjugate of $q$ is given by $q^*=Re(q)-Im(q)$ and its norm is $|q|=\sqrt{qq^*}$. Two quaternions $q_1,q_2\in\Hq$ are called similar, and we write $q_1\sim q_2$, if there exists $s \in \Hq$ with $|s|=1$ such that $s^{*}q_2 s=q_1$. Similarity is an equivalence relation and the class of $q$ is denoted by $[q]$. A necessary and sufficient condition for the similarity of $q_1$ and $q_2$ is that $Re(q_1)=Re(q_2) \textrm{ and }|Im(q_1)|=|Im(q_2)|$.

We now briefly recall the definition of quaternionic Hilbert space (see \cite[Chapter 9]{CGK} for a detailed account).
A right $\Hq$-module $\Hi$ is called a right pre-Hilbert space if there exists a hermitian inner product $\langle.,.\rangle :\mathcal{H}\times\mathcal{H}\to\Hq$, $(x,y) \to \langle x,y\rangle$, satisfying the following properties:
\begin{itemize}
  \item[$(i)$] $\langle xa+yb,w\rangle=\langle x,w\rangle a+\langle y,w\rangle b$,
  \item[$(ii)$] $\langle x,y\rangle=\langle y,x\rangle^*$,
  \item[$(iii)$] $\langle x,x\rangle\geq 0$, and $\langle x,x\rangle=0$ if and only if $x=0$,
\end{itemize}
 for all $x,y,w\in\mathcal{H}$ and $a, b\in\Hq$.
 It follows that
 \[
 \langle x,ya+wb\rangle={a}^*\langle x,y\rangle + {b}^*\langle x,w\rangle.
 \]
If $\Hi$ is complete with respect to the norm $\|x\|=\sqrt{\langle x,x\rangle}$, $x\in \Hi$, then $\Hi$ is called a right quaternionic Hilbert space.

In this paper, we consider $\Hi$ to be separable, i.e, there exists a countable orthonormal basis $\{e_n:n\in\mathbb{N}\}$. Consequently, every $x\in\Hi$ can be uniquely written as
\[
x=\sum_{n=1}^{\infty}e_n\langle x,e_n\rangle=\sum_{n=1}^{\infty}e_nx_n,
\]
where $x_n=\langle x,e_n\rangle\in\Hq$, for all $n\in\mathbb{N}$. Every separable Hilbert space $\Hi$ is isometrically isomorphic to $\ell^2=\ell^2(\N,\F)=\{(x_n)_n\in\F^{\N}:\sum_n |x_n|^2<\infty\}$ via the map $x\mapsto (\langle x,e_n\rangle)_{n\in\N}$, with $x=(x_n)_n$.

A right linear operator $T$  is a map $T:\mathcal{H}\to\mathcal{H}$ such that $T(x+y)=Tx+Ty$ and $T(xa)=(Tx)a$, for all $x,y\in\mathcal{H}$ and $a\in\F$.
We denote the set of all right linear bounded operators on $\mathcal{H}$ by $\mathcal{B}(\mathcal{H})$.
To have a linear structure on $\mathcal{B}(\mathcal{H})$, the Hilbert space $\Hi$ needs to have a left $\Hq$-module structure (see \cite[Chapter 3]{CGK}).

%

The norm in $\mathcal{B}(\mathcal{H})$ is $\|T\|=\mathrm{sup}\,\{\|Tx\|:\|x\|=1\}$.  The adjoint of $T\in\mathcal{B}(\mathcal{H})$ is the operator $T^*\in\mathcal{B}(\mathcal{H})$ which satisfies $\langle Tx,y\rangle=\langle x,T^*y\rangle$, for all $x,y\in\mathcal{H}$. As in the complex case, an operator $T\in\mathcal{B}(\mathcal{H})$ is said to be self-adjoint if $T^*=T$, anti-self-adjoint if $T^*=-T$, normal if $T^*T=TT^*$ and unitary if $T^*T=TT^*=I$, where $I$ is the identity operator. We denote by $\U(\Hi)$ the group of unitary operators.  An operator $T\in\mathcal{B}(\mathcal{H})$ is said to be compact if $(T x_n)_n$ has a convergent subsequence, for every bounded sequence $(x_n)_n$ of $\mathcal{H}$.
Recall that a  bounded linear operator $T$ is invertible  if, and only if, $T$ has a dense range and $T$ is bounded
from below, the latter meaning that there exists $k > 0$ such that $\|Tx\| \geq k\|x\|$ for every $x \in \Hi$. The group of invertible operators is denoted by $\B(\Hi)^{-1}$. See \cite[Chapter 9]{CGK} for further details on $\mathcal{B}(\Hi)$.

Given $T\in\mathcal{B}(\Hi)$ and $q\in\Hq$, we define the operator $\Delta_q(T):\Hi\longrightarrow \Hi$ by $\Delta_q(T)=T^2-2Re(q)T+|q|^2 I$.
Clearly, $\Delta_q(T)$ is a bounded linear operator. The S-spectrum of $T$, $\sigma_S(T)$, is defined by
\begin{eqnarray*}
\sigma_S(T)   &=& \{q\in \Hq: \Delta_q(T)\notin \B(\Hi)^{-1}\},\\
  &=& \Big\{q\in \Hq : \ker(\Delta_q(T))\neq \{0\} \;\text{or}\; \Ran(\Delta_q(T))\neq \Hi\Big\}.
\end{eqnarray*}
The S-spectrum $\sigma_S(T)$ is a compact nonempty subset of $\Hq$ and it is always contained in the closed ball of radius $\|T\|$ around origin $\overline{B(0, \|T\|)}$ (\cite[Theorem 9.2.2]{CGK}). One can show that $q\in \sigma_S(T)$ is equivalent to $[q]\subseteq \sigma_S(T)$. 
The set of right eigenvalues is the set
 \[
 \sigma_r(T)=\{q\in \Hq : Tx=xq, \, x\in\Hi\backslash \{0\}\}.
 \]
 Clearly, $\sigma_r(T)\subseteq \sigma_S(T)$ (see \cite[Proposition 3.1.9]{CGK}).
 In particular, if $\Hi$ is finite dimensional then $\sigma_r(T)=\sigma_{S}(T)$.


We end this Section by recalling that the quaternionic numerical range of $T\in\B(\Hi)$ is the subset of $\Hq$ given by:
\[
W(T)=\{\langle T{x}, {x}\rangle : {x}\in \bS_{\Hi}\},
\]
where $\bS_{\Hi}=\{x\in\Hi: \|x\|=1\}$ is the unit sphere of $\bS_\Hi$. Note that $\bS_{\Hi}$ is the boundary of the unit ball $B(0,1)\subset\mathcal{H}$.
It follows from the definition that the numerical range of $T$ is invariant under unitary equivalence, that is,  $W(T)=W(U^*TU)$, for every $U\in \U(\Hi)$, that $W(T)$ is contained in the closed ball, $\overline{B(0, \|T\|)}$, and when $T$ is self-adjoint, $W(T)\subset\R$.

We always have $\sigma_r(T)\subseteq W(T)$, for every bounded operator $T\in B(\Hi)$.
In fact, if $\lambda\in\sigma_r(T)$ and ${x}\in\bS_{\Hi}$ is a corresponding eigenvector, if we write $Tx=x\lambda$, then $\langle T{x}, {x}\rangle=\langle {x}\lambda, {x}\rangle={\lambda}\in W(T)$. Hence, $W(T)$ contains all right eigenvalues of $T$.




A simple computation shows that  $q\in W(T)$ is equivalent to $[q]\subseteq W(T)$. Therefore, it is enough to study the subset of complex elements in each similarity class, i.e, the bild $B(T)$ of $T$,
\[
B(T)=W(T)\cap \C.
\]

Au-Yeung found necessary and sufficient conditions for the convexity of $W(T)$ (see \cite{Ye1, Ye2}).
Although the bild may not be convex, the upper bild $B^+(T)=W(T)\cap \C^+$  always is. The convexity of the upper bild was shown in the particular case where $T$ is a $n\times n$ quaternionic matrix. The proof reduces to the $2\times 2$ case, see \cite[Lemma 2.4]{ST}. A similar reasoning can be applied to the infinite dimensional case.

\begin{theorem}
Let $T\in \B(\Hi)$. The upper bild $B^+(T)$ is convex.
\end{theorem}

\section{$S$-Spectrum and quaternionic numerical range}

We start this Section with the  notion of approximate unitary equivalence.  In view of Proposition \ref{S-spc and nr closed under a.u.} below, it is enough to characterize the S-spectrum and the closure of the numerical range up to approximate unitary equivalence.

\begin{definition}
We say that $T, R\in \B(\Hi)$ are approximately unitarily equivalent, and we write $T\sim_a R$, if there exists a sequence $(U_n)_{n\in\N}\subset\U(\Hi)$  such that $\lim_{n \rightarrow \infty}\|U_nRU_n^*-T\|=0.$
\end{definition}
From the group structure of $\U(\Hi)$, it can easily be shown that $\sim_a$ is an equivalence relation on $\B(\Hi)$.


Let $\B(\Hi)$ be endowed with the metric induced by the uniform topology. A sequence $(T_n)_n\subset\B(\Hi)$ converges to $T\in\B(\Hi)$ uniformly, and we write $T_n\to T$ if, and only if, $\|T_n-T\|\to 0$.

An immediate consequence is the following property for the S-spectrum.

\begin{proposition}\label{prop semicontinuity S-spec}
Given a sequence $(T_n)_n\subset\mathcal{B}(\mathcal{H})$, if $T_n\to T$ then
\[
\limsup\,\sigma_S(T_n)\subseteq\sigma_S(T).
\]
\end{proposition}
\begin{proof}
We begin by recalling that $\limsup \sigma_S(T_n)= \overline{\cap_{k=1}^{+\infty} \cup_{n=k}^\infty \sigma_S(T_n)}$. Let $q \in \limsup \sigma_S(T_n)$. There is a sequence $(n_k)_k \subset \N$ such that $q_{n_k} \in \sigma_S(T_{n_k})$ and  $q_{n_k} \to q$. Since $T_{n_k} \to T$ then $\Delta_{q_{n_k}}(T_{n_k}) \to \Delta_{q}(T)$.  We know that $q_{n_k} \in \sigma_S(T_{n_k})$, or, in other words, that $\Delta_{q_{n_k}}(T_{n_k}) $ is a non invertible operator. Since the set of non invertible operators is closed then $\Delta_{q}(T)$ is a non invertible operator or, again in other words, that $q \in \sigma_S(T)$.
\end{proof}

We are now in condition to show that the S-spectrum and the closure of the numerical range are invariant under approximate unitary equivalence.

\begin{proposition}\label{S-spc and nr closed under a.u.}
Let $T, R\in \B(\Hi)$ and suppose that $T\sim_a R$. Then,
\begin{itemize}
  \item [$(i)$] $\sigma_S(T)=\sigma_S(R)$;
  \item [$(ii)$] $\overline{W(T)}=\overline{W(R)}$.
\end{itemize}
\end{proposition}
\begin{proof}
Let $(U_n)_n$ be a sequence of unitary operators such that $U_nRU_n^*\to T$.

From Proposition \ref{prop semicontinuity S-spec}, we have
$
\sigma_S(R)=\sigma_S(U_nRU_n^*)=\limsup\sigma_S(U_nRU_n^*)\subseteq\sigma_S(T)
$
and equality (i) follows by symmetry.

 To prove (ii),  fix $\lambda\in W(T)$ and choose a vector $u\in\bS_{\Hi}$ such that $\lambda=\langle Tu, u\rangle$. We have
\[
\lim_{n \rightarrow \infty}|\lambda-\langle RU_n^*u, U_n^*u\rangle|=\lim_{n \rightarrow \infty}|\langle (T-U_nRU_n^*)u,u\rangle|=0.
\]
Since $U_n^*u$ is a unit vector for all $n\in\N$, $\lambda\in \overline{W(R)}$, and since $\lambda\in W(T)$ was arbitrary, $\overline{W(T)} \subseteq\overline{W(R)}$. The equality follows by symmetry.
\end{proof}

For every orthonormal basis $\mathcal{E}=\{e_n: n\in\mathbb{N}\}$ of $\Hi$, there is a decomposition of $\Hi$ defined as follows. Let $\mathcal{H}_{\C}=\overline{\mathrm{span}_{\C}\{\mathcal{E}\}}$ denote the right $\C$-vector space with orthonormal basis $\mathcal{E}$. Every $u\in\Hi$, can be written as
\begin{align*}
  u & =\sum_{n\in\N}e_n\langle u,e_n\rangle= \sum_{n\in\N}e_nu_n\\
   & = \sum_{n\in\N}e_n(x_{n}+y_{n}j)= \sum_{n\in\N}e_nx_{n}+\sum_{n\in\N}e_ny_{n}j \\
   & =x+yj,
\end{align*}
where $u_n=\langle u,e_n\rangle\in\Hq,$ and $x_{n}=\langle x,e_n\rangle,y_{n}=\langle y,e_n\rangle\in\C$ ($n\in\N$). Hence, there is a decomposition with respect to the basis $\mathcal{E}$
\begin{equation}\label{decomposition H}
\Hi=\Hi_{\C}\oplus\Hi_{\C}j,
\end{equation}
where $\Hi_{\C}$ is a $\C$-vector space, and every vector $u\in\Hi$ decomposes uniquely as a sum $u=x+yj,$ for some $x,y\in\Hi_{\C}$. For convenience,
we sometimes write $\Hi$ multiplicatively as $\Hi_{\C}^2$, and represent $u\in\Hi$ as a pair $(x,y)\in\Hi_{\C}^2$. It should be noted that neither $\Hi_{\C}$ nor $\Hi_{\C}j$ are $\Hq$-vector spaces. Here, $\oplus$ is used as a direct sum of real spaces.

Moreover, the $\Hq$-inner product $\langle .,.\rangle$ on $\Hi$ restricts to a $\C$-inner product on $\Hi_{\C}$ denoted $\langle .,.\rangle_{\C}$. In fact, given $x,y\in\Hi_{\C}$, we have:
\begin{equation*}
  \langle x,y\rangle_{\C}  :=\Big\langle\sum_{n}e_nx_n,\sum_{m}e_my_m\Big\rangle =\sum_{n}y^*_nx_n\in\C.
\end{equation*}

Endowed with the inner product $\langle .,.\rangle_{\C}$, $\Hi_{\C}$ becomes a complex Hilbert space and we may relate the norm of a vector $u=x+yj$ in $\Hi$ with the norms of the correspondent vectors $x,y$ in $\Hi_{\C}$
\begin{equation}\label{relating norms}
\|x+yj\|^2=\|x\|^2+\|y\|^2,
\end{equation}
since $\langle x,yj\rangle=-\langle yj,x\rangle$.

The notion of complex operator on a quaternionic Hilbert space $\Hi$, which we now introduce, is central in our work. It allows us to define the $\mathbb{C}$-spectrum of a complex operator and relate it with the S-spectrum.

\begin{definition}\label{def complex op}
An operator $T\in \B(\Hi)$ is called a complex operator if there is an orthonormal basis $\mathcal{E}=\{e_n: n\in\mathbb{N}\}$ of $\Hi$ such that
\[
\langle T(e_n),e_m\rangle\in\C, \quad \text{for every} \quad n,m\in\mathbb{N}.
\]
\end{definition}

We observe that this class contains the normal operators, up to approximate unitary equivalence.

Let $T\in\mathcal{B}(\Hi)$ be a complex operator with respect to the basis $\mathcal{E}$. Then, there is a decomposition of $\Hi$ as in (\ref{decomposition H}) for the basis $\mathcal{E}$ and we can consider the restriction $T_{|}:=T_{|\Hi_{\C}}:\Hi_{\C}\to\Hi$ of $T$. Of course, this is nothing more than the composition of the maps $\Hi_{\C}\hookrightarrow\Hi_{\C}\oplus\Hi_{\C}j=\Hi\to\Hi, T_{|}=T\circ\iota,$ where $\iota:\Hi_{\C}\hookrightarrow\Hi_{\C}\oplus\Hi_{\C}j$ is the injective map $\iota(x)=x+0j$. Although $T_{|}$ is not a $\Hq$-linear map, since $\Hi_{\C}$ is not a $\Hq$-linear space, we easily conclude that it is a right ${\C}$-linear map on $\Hi_{\C}$. The linearity being easy, let us show that $T_{|}(\Hi_{\C})\subseteq\Hi_{\C}$. Take $x\in\Hi_{\C}$. We have
\[
T_{|}(x)=T(x)= T\Big(\sum_{n}e_nx_n\Big)=\sum_nT(e_n)x_n,
\]
and since $T$ is complex,
\[
\langle T_{|}(x),e_m\rangle=\Big\langle\sum_nT(e_n)x_n,e_m\Big\rangle=\sum_n\langle T(e_n),e_m\rangle x_n\in\C,
\]
for every $e_m\in\mathcal{E}$. Hence, $T_{|}(\Hi_{\C})\subseteq\Hi_{\C}$ and, for every $x+yj\in\Hi=\Hi_{\C}\oplus\Hi_{\C}j$, we have
\[
T(x+yj)=T(x)+T(y)j=T_{|}(x)+T_{|}(y)j.
\]

It is well known that one needs $\Hi$ to be a two-sided $\Hq$-vector space in order to define a vector space structure on $\mathcal{B}(\Hi)$, see the discussion in page 136 of \cite{CSS}. However, $\Hi_{\C}$ was defined as a right $\C$-vector space only. It turns out this is enough to have a right $\C$-vector space structure on the set $\mathcal{B}(\Hi_{\C})$ of bounded right $\C$-linear operators on $\Hi_{\C}$. In fact, given $\lambda\in\C$ define for $S\in\mathcal{B}(\Hi_{\C})$,
\[
(S\cdot\lambda)(x):=S(x\lambda).
\]
Then, for every $x,y\in\Hi_{\C}$ and every $\alpha\in\C$,
\begin{align*}
  (S\cdot\lambda)(x+y\alpha) & = S((x+y\alpha)\lambda) = S(x\lambda+y\alpha\lambda)\\
   & = S(x\lambda)+S(y\lambda)\alpha  = (S\cdot\lambda)(x)+(S\cdot\lambda)(y)\alpha.
\end{align*}
An example is the right multiplication by a complex number $\lambda$ which is the right linear operator $I_{\Hc}\cdot\lambda\in\mathcal{B}(\Hi_{\C})$ given by $(I_{\Hc}\cdot\lambda)(x)=x\lambda$ ($x\in\Hi_{\C}$), where $I_{\Hc}$ denotes the identity operator on $\Hi_{\C}$. In particular, if $T\in\mathcal{B}(\Hi_{\C})$ is a complex operator, then for every $\lambda\in\C$, $T-I_{\Hc}\cdot\lambda$ belongs to $\mathcal{B}(\Hi_{\C})$.
Next definition of $\C$-spectrum is the natural notion of spectrum in right complex Hilbert spaces.

%
%
%
%

\begin{definition}\label{def C-spectrum}
If $T\in \B(\Hi)$ is a complex operator, the $\C$-spectrum of $T$ is the subset of $\C$ given by
\[
\sigma_{\C}(T):=\sigma_{\C}(T_{|})=\{\lambda\in \C: T_{|}- I_{\Hc}\cdot\lambda \quad \text{is not invertible in} \quad \B(\Hi_{\C})\}.
\]
\end{definition}

Next result characterizes the S-spectrum of a  bounded complex operator
in terms of the $\C$-spectrum.

\bigskip

\begin{proposition}\label{Prop Sspectrum complex op}
Let $T$ be a complex operator defined as above. Then
\[
\sigma_S(T)=[\sigma_{\C}(T)],
\]
where $[\sigma_{\C}(T)]=\bigcup_{\lambda\in\sigma_{\C}(T)}[\lambda]$.
\end{proposition}

\begin{proof}
In the proof we wil make use of the following equality
\begin{equation}\label{delta_fact}
  \Delta_\lambda(T)x=T(Tx-x\lambda)-(Tx-x\lambda)\lambda^*.
\end{equation}

First we show that $\sigma_S(T)\subseteq [\sigma_{\C}(T)]$. If $\lambda\in \sigma_S(T)$, then $\Delta_\lambda(T)$ is not invertible, that is, $\Delta_\lambda(T)$ is not bounded below or its range is not dense in $\Hi$.

We start by assuming that $\Delta_{\lambda}(T)$ is not bounded below. We will show that $T_{|}-I_{\Hc}\cdot\lambda$ or $T_{|}-I_{\Hc}\cdot\lambda^*$ are not bounded below by contrapositive. Suppose that $T_{|}-I_{\Hc}\cdot\lambda$ and $T_{|}-I_{\Hc}\cdot\lambda^*$ are bounded below. Then there exists $C, C'>0$ such that
\[
\|T_{|}x-x\lambda\|\geq C\|x\| \; \text{and} \; \|T_{|}x'-x'\lambda^*\|\geq C'\|x'\|,\; \forall x, x'\in \Hi_{\C}.
\]
Therefore, using (\ref{delta_fact}), we have
\[
\|\Delta_{\lambda}(T_{|})x\|=\|T_{|}(T_{|}x-x\lambda)-(T_{|}x-x\lambda)\lambda^*\|\geq C'\|T_{|}x-x\lambda\|\geq CC'\|x\|,
\]
and so $\Delta_{\lambda}(T_{|})$ is bounded below.

Now to see that $\Delta_{\lambda}(T)$ is also bounded below, take
$u=x+yj\in\Hi$ and using that $\Delta_\lambda(T)$ is complex we have
\begin{eqnarray}
 \|\Delta_\lambda(T)u\|^2  &=& \|\Delta_\lambda(T_{|})x+\Delta_\lambda(T_{|})yj\|^2 \\ \nonumber
   &=& \|\Delta_\lambda(T_{|})x\|^2+\|\Delta_\lambda(T_{|})yj\|^2 \\ \nonumber
   &\geq & (CC')^2  \left(\|x\|^2+\|y\|^2\right)\\ \label{delta_bb}
    &=& (CC')^2  \|u\|^2.
\end{eqnarray}
So we conclude that, when $\Delta_{\lambda}(T)$ is not bounded below, $T_{|}-I_{\Hc}\cdot\lambda$ or $T_{|}-I_{\Hc}\cdot\lambda^*$ are not bounded below and therefore $\lambda\in\sigma_{\C}(T)$ or $\lambda^*\in\sigma_{\C}(T^*)$.

Now, again by contrapositive, we will show that if $\Delta_\lambda(T)$ has not dense range in $\Hi$, then $T_{|}-I_{\Hc}\cdot\lambda$ or $T_{|}-I_{\Hc}\cdot\lambda^*$ have not dense range in $\Hi_{\C}$.
So suppose that $T_{|}-I_{\Hc}\cdot\lambda$ and $T_{|}-I_{\Hc}\cdot\lambda^*$ have dense range in $\Hi_{\C}$.
Take $z\in\Hi_{\C}$. Then for every $\varepsilon >0$ there is $y\in \Hi_{\C}$ such that $\|z-(T_{|}-I_{\Hc}\cdot\lambda^*)y\|<\varepsilon$. On the other hand, for such $y\in\Hi_{\C}$, there is $x\in\Hi_{\C}$ such that
\begin{equation}\label{dense}
 \|y-(T_{|}-I_{\Hc}\cdot\lambda)x\|<\varepsilon.
\end{equation}
We have, using (\ref{delta_fact}) and putting $\tilde{y}=T_{|}x-x\lambda$,
\begin{eqnarray*}
 \|z-\Delta_{\lambda}(T_{|})x\|  &=& \|z-(T_{|}-I_{\Hc}\cdot\lambda^*)\tilde{y}\|  \\
  &\leq & \|z-(T_{|}-I_{\Hc}\cdot\lambda^*)y\|+\|(T_{|}-I_{\Hc}\cdot\lambda^*)(\tilde{y}-y)\| \\
   &< & \varepsilon +  \|T_{|}-I_{\Hc}\cdot\lambda^*\| \|(\tilde{y}-y)\|  \\
   &< & \varepsilon (1+ \|T_{|}\|+|\lambda|)\;\;\;\; (\text{from (\ref{dense}})). \\
\end{eqnarray*}
Thus, $\Delta_{\lambda}(T_{|})$ has dense range. For $h\in\Hi$ such that $h=h_1+h_2j$, choose $x,y\in\Hi_\C$ such that $\Delta_{\lambda}(T_{|})x$ is close to $h_1$ and $\Delta_{\lambda}(T_{|})y$ is close to $h_2$. Since $\Delta_{\lambda}(T)(x+yj)=\Delta_{\lambda}(T_{|})x+ \Big(\Delta_{\lambda}(T_{|})y\Big)j$, we see that there is $u=x+yj\in\Hi$ such that $\Delta_{\lambda}(T)u$ is close to $h$.

To prove the converse inclusion, it is enough to show that $\sigma_{\C}(T)\subset \sigma_S(T)$.
Let $\lambda\in \sigma_{\C}(T)$ and suppose that $T_{|}-I_{\Hc}\cdot\lambda$ is not bounded below. From continuity of $T_{|}-I_{\Hc}\cdot\lambda^*$, given $\varepsilon>0$, there is $\delta>0$ such that  $\|(T_{|}-I_{\Hc}\cdot\lambda^*)y\|<\varepsilon$, whenever $\|y\|<\delta$. Since $T_{|}-I_{\Hc}\cdot\lambda$ is not bounded below, there is $x\in\Hi_{\C}\backslash\{0\}$ such that $\|(T_{|}-I_{\Hc}\cdot\lambda)x\|\leq \delta \|x\|$. Take $y=\frac{T_{|}x-x\lambda}{\|x\|}$. Then $\|y\|<\delta$ and $\Delta_{\lambda}(T_{|})\frac{x}{\|x\|}=T_{|}y-y\lambda^*$.
Therefore, since  $\|y\|<\delta$,
\[
\|T_{|}y-y\lambda^*\|<\varepsilon \Leftrightarrow \|\Delta_{\lambda}T_{|}x\|<\varepsilon\|x\|
\]
and so $\Delta_{\lambda}T_{|}$ is not bounded below. By a similar reasoning as before, we conclude that $\Delta_{\lambda}T$ is not bounded below.

On the other hand, if the range of $T_{|}-I_{\Hc}\cdot\lambda$ is not dense, that is, $\overline{\Ran (T_{|}-I_{\Hc}\cdot\lambda)}\neq \Hi_{\C}$, since $\overline{\Ran (T_{|}-I_{\Hc}\cdot\lambda)}=(\ker(T_{|}^*-I_{\Hc}\cdot\lambda^*))^\perp$,
we have that $\ker(T_{|}^*-I_{\Hc}\cdot\lambda^*)\neq \{0\}$. Therefore, $T_{|}^*-I_{\Hc}\cdot\lambda^*$ is not bounded below and, in particular, $\lambda^*\in\sigma_{\C}(T_{|}^*)$. Since $\Delta_{\lambda^*}T_{|}^*=\Delta_{\lambda}T_{|}^*$ and from the previous case
we conclude that $\Delta_{\lambda}T_{|}^*$ is not bounded below and $\Delta_{\lambda}T^*$ is also not bounded below. Therefore, $\lambda \in \sigma_S(T^*)$ and from \cite[Proposition 4.7]{GMP} we have $\lambda \in \sigma_S(T)$.

\end{proof}

As an illustration of the above result we compute the S-spectrum  of the quaternionic backward shift operator.

\begin{example}
Let $\Hi$ be the quaternionic Hilbert space $\ell^2$ of square summable sequences. Let $(e_n)_{n}$ be an orthonormal basis of $\ell^2(\N,\C)$ regarded as a basis of $\ell^2(\N,\Hq)$. Let $T\in\B(\ell^2)$ denote the backward shift operator, defined by $T(e_1)=0$ and $T(e_n)=e_{n-1}$ for all $n\geq 2$.
Since  $\sigma(T)=\overline{B_{\C}(0,1)}$ (see \cite[Example 2.3.2]{Mu}), it follows from Proposition \ref{Prop Sspectrum complex op} that $\sigma_S(T)=\overline{B_{\Hq}(0,1)}$. Here, $\overline{B_{\F}(0,1)}$ denotes the closed unit ball in $\F$.
\end{example}

Now we compute the quaternionic numerical range of the backward shift operator. This example shows that, contrary to the finite dimensional case,  $W(T)$ is not closed and therefore not compact, in general. Although this computation is similar to the complex case (\cite[Example 2]{GR}), we include it for the sake of completeness.

\begin{example}
Let $T\in\B(\ell^2)$ be the backward shift operator.
Since $\|T\|=1$, then  $W(T)\subseteq \overline{B(0,1)}$.
Clearly, $s=0$ is an eigenvalue of $T$ and for $0<|s|<1$, $x_{(s)}=\sum_{n\geq 1}e_n s^n\in\Hi$ is an eigenvector of $T$ for $s$, i.e, $Tx_{(s)}=x_{(s)}s$. Hence,
$
B(0,1)\subseteq\sigma_{r}(T)\subseteq W(T)\subseteq\overline{B(0,1)}.
$
To show that the quaternionic numerical range is the open unit ball, suppose that there was $\lambda \in W(T)$  such that $|\lambda|=1$. Then, there would be a unit vector $x\in \bS_{\Hi}$ with $\lambda =\langle Tx,x\rangle$. Since $\|T\|=1$ we would have
\[
1=|\lambda |=|\langle Tx,x\rangle|\leq \|Tx\|\|x\| \leq 1.
\]
The Cauchy-Schwarz inequality implies that  $Tx=x\lambda$.
Let $x=\sum_{n\geq 1} e_n a_n\in \ell^2$. It follows that
$
\sum_{n\geq 1} e_n a_{n+1}=\sum_{n\geq 1} e_n a_n \lambda.
$
Then $|a_{n}|=|a_1|$, for all $n\in \N$, which is impossible since $x\in \ell^2$ and therefore $\lambda$ cannot belong to $W(T)$. We conclude that $W(T)={B(0,1)}$.
\end{example}

Contrary to the finite dimensional case, where we always have $\sigma_S(T)=\sigma_r(T)\subseteq W(T)$ for every $n\times n$ matrix $T$, for infinite dimensional Hilbert spaces this is no longer true, as the above examples shows since $\sigma_S(T)=\overline{B(0,1)}$ but $W(T)=B(0,1)$. However, the result holds true for the closure of $W(T)$ for every operator $T\in\B(\Hi)$.

The next two results appeared first in \cite{MBB}. We include them with a proof for convenience of the reader.

\begin{theorem}\label{Theo_Sspecturm_NR}
Let $T\in \B(\Hi)$. Then $\sigma_S(T)\subseteq \overline{W(T)}$.
\end{theorem}

\begin{proof}
Let $q\in\sigma_S(T)$. Then $\Delta_q(T)$ is not invertible, that is, $\Delta_q(T)$ is not bounded
from below or $\Delta_q(T)$ does not have dense range.

First suppose that $\Delta_q(T)$ is not bounded from below. Then,
 there exists a sequence of vectors $x_n\in \bS_{\Hi}$ such that $\lim_{n \rightarrow\infty}\|\Delta_q(T)x_n\|=0$. If $Tx_n-x_nq=0$ then $q$ is a right eigenvalue of $T$ and therefore $q\in W(T)$. Suppose $Tx_n-x_nq\neq 0$.
Using the Cauchy-Schwarz inequality, we have
 \[
 |\langle \Delta_q(T)x_n, Tx_n-x_nq \rangle| \leq \|\Delta_q(T)x_n\|\, \|Tx_n-x_nq\|
 \]
and therefore
\[
\lim_{n \rightarrow\infty} \langle \Delta_q(T)x_n, Tx_n-x_nq \rangle=0.
\]
Taking into account that $\Delta_q(T)x_n$ can be written in the form
\[
\Delta_q(T)x_n=T(Tx_n-x_nq)-(Tx_n-x_nq)q^*,
\]
 we get
\[
\lim_{n \rightarrow\infty} \langle Ty_n-y_nq^*, y_n \rangle=0,
\]
where $y_n=\frac{Tx_n-x_nq}{\|Tx_n-x_nq\|}$.
So $\lim_{n \rightarrow\infty}\langle Ty_n, y_n \rangle - q^*=0$. As each  $\langle Ty_n, y_n \rangle\in W(T)$, then $q^*\in \overline{W(T)}$ and so $q\in\overline{W(T)}$.

Now suppose that   $\Delta_q(T)$ is bounded  from below but $\Delta_q(T)$ does  not have dense range. From \cite[Theorem 9.1.14]{CGK} we know that $\Ran(\Delta_q(T))^{\bot}=\ker (\Delta_q(T^*))$. So $\overline{\Ran(\Delta_q(T))}=\ker (\Delta_q(T^*))^{\bot}\neq\Hi$. Therefore $\ker (\Delta_q(T^*))\neq \{0\}$. From Proposition 3.1.9 of \cite{CGK} we have that $q$ is a right eigenvalue of $T^*$. Since $W(T^*)=W(T)$, we conclude that $q\in W(T)$.
\end{proof}


This result allow us to establish a relation between the S-spectral radius and the numerical radius of a bounded linear operator in $\Hi$.
Recall that, for $T \in \B(\Hi)$ the S-spectral radius of $T$  is defined to be the nonnegative real number
\[
r_S(T):=\sup\{|q|: q\in\sigma_S(T)\}.
\]
It is known that $r_S(T)\leq \|T\|$ (see \cite[Theorem 4.3]{GMP}).
The quaternionic numerical radius of $T\in\B(\Hi)$ is defined by
\[
\omega (T):=\sup \{|\lambda|: \lambda\in W(T)\}.
\]
The quaternionic numerical radius shares some properties of the complex numerical radius, namely:
\begin{equation}\label{propNradius}
  \omega(T+S)\leq \omega(T)+\omega(S)\quad \text{and} \quad \omega( T)=\omega( T^*),
\end{equation}
for $T, S\in \B(\Hi)$. However, $\omega(\cdot)$ fails to be a norm even for finite dimensional quaternionic Hilbert space (see \cite{CDM5}). But we still have the following relation between the spectral radius, numerical radius and the operator norm.

\begin{corollary}
Let $T\in \B(\Hi)$. Then
\[
r_S(T)\leq\omega (T)\leq\|T\|\leq 2 \omega (T).
\]
\end{corollary}

\begin{proof}
It follows from the previous result and from the Cauchy-Schwarz inequality that, for any $T\in \B(\Hi)$, $
r_S(T)\leq\omega (T)\leq\|T\|.$
It remains to see that $\|T\|\leq 2 \omega (T)$. In fact, we can write
\[
\|T\|\leq\frac{\|T+T^*\|}{2}+\frac{\|T-T^*\|}{2}
\]
and since $T+T^*$ and $T-T^*$ are normal operators, from \cite[Theorem 3.3]{R2} we have
\[
\|T\|\leq\frac{1}{2}\omega({T+T^*})+\frac{1}{2}\omega({T-T^*}).
\]
From (\ref{propNradius}) and since $\omega(\alpha T)=|\alpha|\omega (T)$, for $\alpha\in \R$, we conclude that
$\|T\|\leq 2\omega({T})$.
\end{proof}

\section{Quaternionic numerical range of a complex operator}

In this Section, we consider $T\in\B(\Hi)$ to be a complex operator as in definition \ref{def complex op}.
 Every such operator can be uniquely written as a sum of a self-adjoint operator and an anti-self-adjoint operator
 \begin{eqnarray}
   T &=&  \frac{1}{2}(T+T^*)+\frac{1}{2}(T-T^*)\label{opTcomplex}\\
    &=& \widetilde{H}+\widetilde{S}, \nonumber
 \end{eqnarray}
with $\widetilde{H}$ and $\widetilde{S}$ complex normal operators.

%

We aim to characterize the quaternionic numerical range of a complex operator. By similarity of the elements of $W(T)$, it is enough to
 characterize the bild. From this, we obtain the shape of the upper bild in terms of the complex numerical range and two reals values, the infimum and supremum of the real elements in que quaternionic numerical range. The finite dimensional case was treated in \cite{CDM5}.
\begin{proposition}\label{prop bild complex op}
Let $T\in\B(\Hi)$ be a complex operator. Then,
\[
B(T)=\{\langle Tx,x\rangle+\langle T^*y,y\rangle:(x,y)\in\bS_{\Hi_{\C}^2},\langle (T-T^*)y,x\rangle=0\}.
\]
\end{proposition}
\begin{proof}
For an element $\omega$ in $W(T)$, there is $u=x+yj\in\bS_{\Hi}$ such that
\begin{eqnarray*}
  \omega &=& \langle Tu,u\rangle \\
   &=& \langle T(x+yj),(x+yj)\rangle\\
   &=& \langle Tx,x\rangle -j\langle Tx,y\rangle+\langle Ty,x\rangle j - j\langle Ty,y\rangle j.
\end{eqnarray*}

Since $\langle Tx,y\rangle, \langle Ty,y\rangle\in\C$, it follows from \cite[Theorem 2.1]{Zh} that
\begin{eqnarray*}
  \omega &=& \langle Tx,x\rangle -\langle Tx,y\rangle^* j+\langle Ty,x\rangle j +\langle Ty,y\rangle^* \\
   &=& \langle Tx,x\rangle +\langle y,Ty\rangle+\left(\langle Ty,x\rangle-\langle T^*y,x\rangle\right)j \\
   &=& \langle Tx,x\rangle +\langle T^*y,y\rangle+\langle (T- T^*)y,x\rangle j.
\end{eqnarray*}
Since $T-T^*$ is a complex operator and $x,y\in\Hi_{\C}$, then $\langle (T- T^*)y,x\rangle j\in\Span\{j,k\}$.
Hence, $\omega\in B(T)$ if, and only if, $\langle (T- T^*)y,x\rangle=0$ and $(x,y)\in\bS_{\Hi_\C^2}$,
and we conclude that $\omega=  \langle Tx,x\rangle +\langle T^*y,y\rangle$.
\end{proof}
From Proposition \ref{prop bild complex op}, an element $w\in B(T)$ is of the form
\begin{eqnarray*}
  w &=& \langle Tx,x\rangle+\langle T^*y,y\rangle \\
   &=& \alpha^2 \langle Tx_{\bS},x_{\bS}\rangle+(1-\alpha^2)\langle T^*y_{\bS},y_{\bS}\rangle,
\end{eqnarray*}
where $(x,y)\in\bS_{\Hi_{\C}^2}$, $x=\|x\|x_{\bS}, y=\|y\|y_{\bS}$ and $\alpha^2=\|x\|^2$. It follows that
\[
B(T)\subset\mathrm{conv}\,\{W_{\C}(T),W_{\C}(T^*)\},
\]
where $W_{\C}(T)$ denotes the complex numerical range of $T$.
An immediate consequence is that if ${W_{\C}(T)}={W_{\C}(T^*)}$ then ${B(T)}={W_{\C}(T)}$. The converse is also true. In fact, if ${B(T)}={W_{\C}(T)}$, using that $(W_\C(T))^*=W_\C(T^*)$, we have $(B(T))^*=W_\C(T^*)$. Therefore, $B(T)=W_\C(T^*)$ and so  ${W_{\C}(T)}={W_{\C}(T^*)}$. In addition, from Toeplitz-Hausdorff Theorem we conclude that $B(T)$ is convex.
Therefore we can establish a result that also holds for finite dimensional Hilbert space, see \cite[Corollary 3.8]{CDM5}).

\begin{corollary}
Let  $T\in\B(\Hi)$ be a complex operator. Then $W_{\C}(T)=W_{\C}(T^*)$ if, and only if, $B(T)=W_{\C}(T)$. Moreover, $B(T)$ is convex.
\end{corollary}


Next result characterizes the upper bild of a complex operator $T\in\B(\Hi)$ in terms of the (upper) complex numerical ranges of $T$ and $T^*$, and two real numbers $\underline{v}$ and $\overline{v}$. In fact,
since $B(T)\cap\R$ is a bounded, convex set and thus an interval in $\R$, we may define
\begin{eqnarray*}
  \underline{v} & = & \mathrm{inf}\,B(T)\cap\R \\
  \overline{v} & = & \mathrm{sup}\,B(T)\cap\R.
\end{eqnarray*}

\begin{theorem}\label{theo upper bild}
Let $T\in\B(\Hi)$ be a complex operator. Then,
\[
\overline{B^+(T)}=\mathrm{conv}\,\{\overline{W^+_{\C}(T)},\overline{W^+_{\C}(T^*)},\underline{v},\overline{v}\},
\]
where $\underline{v}=\mathrm{inf}\,B(T)\cap\R$ and $\overline{v}=\mathrm{sup}\,B(T)\cap\R$.
\end{theorem}

\begin{proof}
From $W_{\C}(T)\subseteq W_{\Hq}(T)$ we have that $W_{\C}^+(T)\subseteq B^+(T)$. On the other hand, from $(W_{\Hq}(T))^*=W_{\Hq}(T)$
we have $W_{\C}^+(T^*)\subseteq B^+(T)$. Since the closure of the upper bild $\overline{B^+(T)}$ is convex and contains $\underline{v}$ and $\overline{v}$, we conclude that
\[
\mathrm{conv}\,\{\overline{W_{\C}^+(T)},\overline{W_{\C}^+(T^*)},\underline{v},\overline{v}\}\subseteq \overline{B^+(T)}.
\]

To prove the converse, let $w\in \overline{B^+(T)}$. Then, $w=\lim_k w_k$, for some sequence $(w_k)_k$ in $B^+(T)$. Since $w_k\in B^+(T)$ for every $k\in\N$, from Proposition \ref{prop bild complex op} we know that for some $(x_k,y_k)\in\bS_{\Hi_{\C}^2}$ such that $\langle (T-T^*)y_k,x_k\rangle=0$ we have
\begin{eqnarray*}
 w_k  &=& \langle Tx_k,x_k\rangle+\langle T^*y_k,y_k\rangle \\
   &=& \|x_k\|^2 \langle Tx_{{\bS},k},x_{{\bS},k}\rangle+\|y_k\|^2 \langle T^*y_{{\bS},k},y_{{\bS},k}\rangle\\
   &=&  \alpha_k\omega_{1,k}+(1-\alpha_k)\omega_{2,k},
\end{eqnarray*}
with $\alpha_k\in[0,1]$, $w_{1,k}=\langle T x_{{\bS},k},x_{{\bS},k}\rangle$ and $w_{2,k}=\langle T^* y_{{\bS},k},y_{{\bS},k}\rangle$.

Note that $(w_{1,k})_k\subset \overline{W_{\C}(T)}$ and  $(w_{2,k})_k\subset \overline{W_{\C}(T^*)}$. Hence, we have convergent subsequences, say,
$(w_{1,k_m})_m\to w^1$, $(w_{2,k_m})_m\to w^2$ and $(\alpha_{k_m})_m\to \alpha$.
Thus, $w_{k_m}=\alpha_{k_m}w_{1,k_m}+(1-\alpha_{k_m})w_{2,k_m}$ converges to $w=\alpha w^1+(1-\alpha)w^2$. In other words, considering if necessary subsequences, we see that $w$ is a convex combination of $w^1=\lim_k w_{1,k}$ and $w^2=\lim_k w_{2,k}$. At this point, the proof split into three cases.

\textbf{Case 1.} If $(w_{1,k})_k$ lies in $\overline{W_{\C}^+(T)}$ and $(w_{2,k})_k$ lies in $\overline{W_{\C}^+(T^*)}$ then  $w_k$ lies in $\mathrm{conv}\,\{\overline{W^+_{\C}(T)},\overline{W^+_{\C}(T^*)},\underline{v},\overline{v}\}$, and so does $w=\lim_k w_k$. This follows from the fact that the convex hull of bounded closed sets in $\C$ is closed.

\textbf{Case 2.} Suppose $(w_{1,k})_k$ lies in $\overline{W_{\C}^-(T)}$ and $(w_{2,k})_k$ lies in $\overline{W_{\C}^+(T^*)}$.
Take a subsequence $(w_{1,k},w_{2,k})_k$ such that $\{w_{1,k}, w_{2,k}\}\nsubseteq\R$. If such subsequence does not exist it means that, for a certain $p\in\mathbb{N}$, we have that $w_{1,k},w_{2,k}\in\R$, for every $k>p$. Note that, since $\C^-\cap\R=\C^+\cap\R$ we have $\overline{W_{\C}^-(T)}\cap\R=\overline{W_{\C}(T)}\cap\R=\overline{W_{\C}^+(T)}\cap\R$
and therefore, when $w_{1,k}$ lie in $\R$ necessarily $w_{1,k}\in\overline{W_{\C}^+(T)}\cap\R$. It follows that $w_{1,k},w_{2,k}\in\overline{W_{\C}^+(T)}$ for $k>p$. This was treated in case 1. On the other hand, when such subsequence exists, which we still denote by $(w_{1,k},w_{2,k})_k$ for simplicity, let
$r_k=[w_{1,k},w_{2,k}]\cap\R$ for every $k$.
Since $w_k\in[w_{1,k},w_{2,k}]$ is an element of the upper bild, then $w_k\in[r_k,w_{2,k}]$. We observe that $[w_{1,k},w_{2,k}]$ is contained in the bild $\overline{B(T)}$. In fact, an element of $[w_{1,k},w_{2,k}]$ is of the form $\alpha_kw_{1,k}+(1-\alpha_k)w_{2,k}$, for some $\alpha_k\in[0,1]$, where $w_{1,k}=\langle T z_{1,k},z_{1,k}\rangle$ and $w_{2,k}=\langle T^* z_{2,k},z_{2,k}\rangle$, for $z_{1,k},z_{2,k}\in\bS_{\Hi_{\C}}$ and $\langle (T-T^*)z_{2,k},z_{1,k}\rangle=0$.
Now, simply take $x_k=\sqrt{\alpha_k}z_{1,k}$ and $y_k=\sqrt{1-\alpha_k}z_{2,k}$, in Proposition \ref{prop bild complex op} and one see that $[w_{1,k},w_{2,k}]\subset \overline{B(T)}$. Hence, $r_k\in[\underline{v},\overline{v}]$. Therefore, $w_k$ can be rewritten as a convex combination of $w_{2,k}$, $\underline{v}$ and $\overline{v}$ and so $w_k$ lies in $\mathrm{conv}\,\{\overline{W_{\C}^+(T^*)},
\underline{v},\overline{v}\}$. Taking the limit it follows that $w\in\mathrm{conv}\,\{\overline{W_{\C}^+(T^*)},
\underline{v},\overline{v}\}\subseteq \mathrm{conv}\,\{\overline{W_{\C}^+(T)},\overline{W_{\C}^+(T^*)},\underline{v},\overline{v}\}$.

\textbf{Case 3}. When $(w_{1,k})_k$ lies in $\overline{W_{\C}^+(T)}$ and $(w_{2,k})_k$ lies in $\overline{W_{\C}^-(T^*)}$, is similar to Case 2.
\end{proof}

\section{$S$-Spectrum and numerical range of a normal operator}

From Theorem \ref{Theo_Sspecturm_NR} we have that $\sigma_S(T)\subseteq \overline{W(T)}$, for every $T\in \B(\Hi)$. It follows that  $\sigma_S^+(T)\subseteq \overline{B^+(T)}$, where $\sigma_S^+(T):=\sigma_S(T)\cap\C^+$. Since $\underline{v}, \overline{v}\in \overline{B^+(T)}$ and the upper bild is a convex set, we conclude that
\[
\Conv\{\sigma_S^+(T), \underline{v}, \overline{v}\}\subseteq \overline{B^+(T)}.
\]
When $T$ is a quaternionic
normal operator, the above inclusion is an equality, as we now prove.
We begin by showing that every quaternionic normal operator $T\in\B(\Hi)$ is approximately unitarily equivalent to a certain complex diagonal operator $D\in\B(\Hi)$.

\begin{proposition}\label{prop normal app diagonal}
Let $T\in\B(\Hi)$ be a normal operator. Then there exists a diagonal operator $D\in \B(\Hi)$ with respect to an orthonormal basis $\{e_n:n\in\N\}$ of $\Hi$, such that $T\sim_a D$, where $D(e_n)=e_nd_n$, and $d_n\in\C^+$.
\end{proposition}

\begin{proof}
By  the Weyl-von Neumann-Berg Theorem for quaternionic operators \cite[Theorem 3.4]{R1},  there exists a diagonal operator $\tilde{D}\in\B(\Hi)$ and a compact operator $K\in\B(\Hi)$ with $\|K\|<\epsilon$ such that $T=\tilde{D}+K$, for every $\epsilon >0$. It follows that $T\sim_a \tilde{D}$.

Let, for every $n\in \N$, $\tilde{D}e_n=e_n\tilde{d}_n$ and let $u_n$ be the element of $\bS_{\bP}$ such that $u_n^*\tilde{d}_nu_n\in \C^+$. Let $U\in \U(\Hi)$ be given by $Ue_n=e_nu_n$. It follows that $\tilde{D}$ is unitarily equivalent to the diagonal operator $D=U^*\tilde{D}U$. Moreover, $De_n=e_nd_n$, for every $n\in \N$, with $d_n=u^*_n\tilde{d}_nu_n\in\C^+$. By transitivity $T\sim_a D$.
\end{proof}

\begin{theorem}\label{theo_upperbild_spectrum}
Let $T\in \B(\Hi)$ be a normal operator. Then
\[
\overline{B^+(T)}=\Conv \{\sigma_{S}^+(T), \underline{v}, \overline{v}\},
\]
where $\underline{v}=\inf B(T)\cap \R$ and $\overline{v}=\sup B(T)\cap \R$.
\end{theorem}

\begin{proof}
It remains to prove that
\[
\overline{B^+(T)}\subseteq
\Conv\{\sigma_S^+(T), \underline{v}, \overline{v}\}.
\]
From Proposition \ref{prop normal app diagonal} there exists a diagonal operator $D\in\B(\Hi)$ such that $T\sim_a D$. By Proposition \ref{S-spc and nr closed under a.u.}, we have $\overline{W(T)}=\overline{W(D)}$ and so $\overline{B^+(T)}=\overline{B^+(D)}$.

Since $D$ is a complex operator as defined in Proposition \ref{prop normal app diagonal}, we have $\overline{W_\C(D)}\subset \C^+$, $\overline{W_\C(D^*)}\subset \C^-$. Hence, Theorem \ref{theo upper bild} implies that
\[
\overline{B^+(D)}=
\Conv\{\overline{W_\C(D)}, \underline{v}, \overline{v}\}.
\]
From \cite[Theorem 1.4-4]{GR} we know that $\overline{W_\C(D)}=\Conv\{\sigma_{\C}(D)\}$.
Since $\Conv\{\sigma_{\C}(D)\}=\Conv\{\sigma_S^+(D)\}$ and $\sigma_S(D)=\sigma_S(T)$ (see Proposition \ref{S-spc and nr closed under a.u.}) the result follows.
\end{proof}

Now we characterize $\underline{v}$ and $\overline{v}$ for quaternionic normal operators $T\in\B(\Hi)$. We will see that these  two real values are
constructed from pairs of eigenvalues.
As  mentioned above, $T\sim_a D$, where $De_n=e_nd_n, d_n=h_n+s_ni\in\C^+$, for every $n\in \N$, and $\overline{W(T)}=\overline{W(D)}$.
Since $D$ is a complex operator, the decomposition (\ref{opTcomplex}) can be written, upon restriction to $\Hi_{\C}$, in the form
\[
D=\widetilde{H}+\widetilde{S}=H+S\cdot i,
\]
where $H=\widetilde{H}_{|}$  and $S=-\widetilde{S}_{|}\cdot i$ are diagonal operators. Specifically, $He_n=e_nh_n$ and $Se_n=e_ns_n$, with $h_n\in\bR$ and $s_n\in\bR_0^+$. Clearly, we have $D-D^*=2S\cdot i$.
Thus, from Proposition \ref{prop bild complex op} we may write the bild of $D$ in the form
\begin{eqnarray*}
  B(D) &=&  \{\langle (H+S\cdot i)x, x\rangle+ \langle (H-S\cdot i)y, y\rangle\,:\; (x,y)\in\bS_{\Hi_\C^2},\; \langle Sy, x\rangle=0\}\\
   &=&   \{\langle Hx, x\rangle+ \langle Hy, y\rangle+\left(\langle Sx, x\rangle-\langle Sy, y\rangle\right)i\,:\; (x,y)\in\bS_{\Hi_\C^2},\; \langle Sy, x\rangle=0\}.
\end{eqnarray*}

So, the real elements of the bild are of the form
\[
  B(D)\cap\bR = \{\langle Hx, x\rangle+ \langle Hy, y\rangle\,:\; (x,y)\in\bS_{\Hi_\C^2},\; \langle Sy, x\rangle=0, \; \langle Sx, x\rangle=\langle Sy, y\rangle\}.
\]

Therefore, in order to characterize $\underline{v}$ (and $\overline{v}$) one needs to find the infimum (supremum) of the real function
\begin{eqnarray}
f(x,y) &=& \langle Hx,x\rangle + \langle Hy,y\rangle\label{f}\\
   &=& \sum_{k=1}^\infty h_k\Big(|x_k|^2+|y_k|^2\Big)\nonumber
\end{eqnarray}
for $(x,y)$ subject to\\
\begin{center}
  (I) \;$\langle Sx, x\rangle=\langle Sy, y\rangle$,\quad\quad
(II)\;$\langle Sy, x\rangle=0$,\quad\quad
(III) \;$(x,y)\in\bS_{\Hi_\C^2}$.\\
\end{center}
These pairs $(x,y)$ define a domain $\Omega\subset \Hi$. In fact, $(x,y)\in \Hi_{\C}^2$ verifies $(I), (II), (III)$ if, and only if, $(x,y)$ is in the fiber $\Omega=\varphi^{-1}(0)$, where
$\varphi=(\varphi_1,\varphi_2,\varphi_3)$ and
$\varphi_i:\Hi_{\C}^2 \longrightarrow\C \quad (i=1,2,3)$ are given by
\begin{eqnarray*}
  \varphi_1(x,y) &=& \langle Sx, x\rangle - \langle Sy, y\rangle= \sum_{k=1}^\infty s_k\Big(|x_k|^2-|y_k|^2\Big)\\
  \varphi_2(x,y) &=& \langle Sy, x\rangle= \sum_{k=1}^\infty s_kx_k^*y_k \\
 \varphi_3(x,y) &=& \langle x, x\rangle+ \langle y, y\rangle-1= \sum_{k=1}^\infty |x_k|^2+|y_k|^2-1.
\end{eqnarray*}

%
%

 The strategy of the proof reduces to the numerical range of normal $N\times N$ matrices (see \cite{CDM4}), followed by a passage to the limit. Concretely, we begin by
 considering the $N\times N$ diagonal matrix $D_N=\diag\{d_1, d_2,\dots, d_N\}$ and  proving that $\vmin$ is smaller than $\vmin_N$, where $\vmin_N$ is the minimum of $B(D_N)\cap \R$, that is, the minimum of the function
  \[
  f_N:\Omega_N\longrightarrow \bR
  \]
 \begin{equation}\label{f_N}
 f_N(x,y)=\sum_{k=1}^N h_k\Big(|x_k|^2+|y_k|^2\Big)
  \end{equation}
  over a domain $\Omega_N\subset \C^{2N}$, defined by
  \begin{equation}\label{OmegaN}
 \Omega_N=\{(x,y)\in {\C}^{2N}\;: \;\varphi^N(x,y)=0\},
  \end{equation}
with $\varphi^N=(\varphi_1^N,\varphi_2^N,\varphi_3^N)$, the $\varphi_i^N$ being the finite analogues of $\varphi_i$ defined above.


 To begin with, we restrict to the case where there is a $j \in \N$ such that $s_jx_j y_j\neq 0 $. The case where $s_jx_j y_j=0$ for any $j \in \N$ will be dealt in the proof of Theorem \ref{theo_vmin_vmax}. For simplicity assume that $j=1$ and therefore we assume now that we are in the case where $s_1|x_1||y_1| > 0$. Note that condition (II) implies that

\begin{equation}\label{conditionii}
0<s_1|x_1||y_1|\leq\sum_{k=2}^\infty s_k|x_k||y_k|.
\end{equation}

We will consider two cases.\\
\textbf{Case 1.} $s_1|x_1||y_1|<\sum_{k=2}^\infty s_k|x_k||y_k|$.
\\

Since $s_1|x_1y_1|<\sum_{k=2}^\infty s_k|x_k||y_k|$, there is an $M$ such that for $N>M$,
\begin{equation}\label{condN}
  0<s_1|x_1y_1|<\sum_{k=2}^N s_k|x_k||y_k|.
\end{equation}

\begin{lemma}\label{lemma_phi2=0}
  Let $(x,y)\in \Omega$.
     If $s_1|x_1y_1|<\sum_{k=2}^\infty s_k|x_k||y_k|$, then there
  is $M\in\N$ such that for $N>M$, there is $(x',y')\in\C^{2N}$ with
\[
|x'_k|=|x_k|,\quad |y'_k|=|y_k|,\quad   k=1,\dots, N,
   \]
  and $\varphi_2^N(x',y')=0$.
    \end{lemma}

  \begin{proof}
Chose $N$ according to (\ref{condN}) and let $(\check{x}, \check{y}), (\hat{x}, \hat{y})\in \bS_{\C^{2N}}$, with
  \begin{eqnarray}
    \check{x}_k=\hat{x}_k=|x_k|\,, & & \check{y}_k  = |y_k|,\quad \quad k=1, \dots, N, \nonumber \\
    \hat{y}_1= |y_1|,\quad\quad\quad  & & \, \hat{y}_k=-|y_k|, \quad k=2, \dots, N.\label{yk}
  \end{eqnarray}

  From condition (\ref{condN}) we have $\varphi_2^N(\check{x}, \check{y})>0$. From (\ref{condN}) and (\ref{yk}), it follows
  that  $\varphi_2^N(\hat{x}, \hat{y})=s_1|x_1||y_1|-\sum_{k=2}^N s_k|x_k||y_k|<0$. On the other hand, path connectedness of $\Pi_{k=1}^N \bS(0,|x_k|)\times \bS(0,|y_k|)$ implies that there is a continuous path
   $\gamma:[0,1]\longrightarrow \Pi_{k=1}^N \bS(0,|x_k|)\times \bS(0,|y_k|)$, joining $(\hat{x}, \hat{y})$ to $(\check{x}, \check{y})$, where
$\gamma(0)=(\hat{x}, \hat{y})$
and
$\gamma(1)=(\check{x}, \check{y}).$
  Since $\varphi_2^N\circ\gamma(0)<0<\varphi_2^N\circ\gamma(1)$,
  by continuity there is a $t_0\in [0,1]$ such that  $\varphi_2^N\circ\gamma(t_0)=0$.
  Take $(x',y')=\gamma(t_0)$. Since each coordinate of $\gamma$ is over a sphere of constant radius, it is clear that $|x'_k|=|x_k|$ and $|y'_k|=|y_k|,  k=1,\dots, N$.

  \end{proof}

We now associate to each vector $z$ in $\ell^2$ or in $\C^n$ a new vector $z_{\parallel}$. If $z\in\ell^2$
the vector is $z_{\parallel}=(|z_1|^2,|z_2|^2, \dots )\in\ell^1$, in the case that $z \in \C^n$ then let $z_\parallel=(|z_1|^2,|z_2|^2, \dots, |z_n|^2, 0,0,\dots,0 )\in\ell^1.$
We will show that, for $x,y\in \ell^2$, $(x_{\parallel}, y_{\parallel})$  can be approximated by
 a vector with finite support which belongs to the domain $\Omega$.

\begin{lemma}\label{lemma_omega_nonempty}
Let $(x,y)\in \Omega$. Then there is $(\hat{x}^\varepsilon, \hat{y}^\varepsilon)\in \Omega$ with finite support
such that
\begin{equation}\label{convergence}
(\hat{x}^\varepsilon_\parallel, \hat{y}^\varepsilon_\parallel) \longrightarrow_{\ell_1} (x_\parallel,y_\parallel), \quad
\text{as}\quad \varepsilon\longrightarrow 0.
\end{equation}
\end{lemma}

\begin{proof}
Fix $\varepsilon>0$ and choose $N>\frac{1}{\varepsilon}$ such that
\begin{equation}\label{series pequenas}
\sum_{k=N+1}^{\infty} s_k|x_k|^2<\frac{\varepsilon}{2} \quad\text{and}\quad
 \sum_{k=N+1}^{\infty} s_k|y_k|^2<\frac{\varepsilon}{2}.
\end{equation}
Note that $N$ exists since $S$ is bounded.
From condition (\ref{condN}) we have that $x\neq 0$ and $y\neq 0$. Pick $M$ from lemma \ref{lemma_phi2=0} and let $N_\varepsilon>\max\{N,M\}$. Then there is $(x',y')\in\C^{2N_\varepsilon}$ such that $\varphi_2^{N_{\varepsilon}}(x',y')=0$. Let
\begin{equation}\label{x_barra_epsilon}
  (\overline{x}^\varepsilon,\overline{y}^\varepsilon)=
  \frac{(\sqrt{1-\alpha_\varepsilon}x',y')}{n_\varepsilon},
\end{equation}
with $n_\varepsilon=\|(\sqrt{1-\alpha_\varepsilon}x',y')\|$ and
$\alpha_\varepsilon=\dfrac{\sum_{k=1}^{N_\varepsilon} s_k\Big(|x_k|^2-|y_k|^2\Big)}
{\sum_{k=1}^{N_\varepsilon} s_k|x_k|^2}.$
Such $\alpha_\varepsilon\in [0,1)$ is well defined since by hypothesis $s_1|x_1|>0$; and $n_\varepsilon \neq 0$ since the vector $(x',y')$ in Lemma \ref{lemma_phi2=0} is chosen in a way that $y'_1= |y_1| >0$.

Moreover, we have that $\alpha_\varepsilon\rightarrow 0$ as
$\varepsilon\rightarrow 0$. In fact, since $(x,y)\in \Omega$, $\varphi_1(x,y)=0$.
From (\ref{series pequenas}) and (I) it follows that
\begin{eqnarray*}
 \left|\sum_{k=1}^{N_\varepsilon} s_k\Big(|x_k|^2-|y_k|^2\Big)\right| & = &
 \left|\sum_{k=1}^{\infty} s_k\Big(|x_k|^2-|y_k|^2\Big)- \sum_{k=N_\varepsilon+1}^\infty
   s_k\Big(|x_k|^2-|y_k|^2\Big)\right| \\
   &\leq & \sum_{k=N_\varepsilon+1}^\infty s_k|x_k|^2 +
    \sum_{k=N_\varepsilon+1}^\infty s_k|y_k|^2\\
   & < & \varepsilon.
\end{eqnarray*}

First we will show that $(\overline{x}^\varepsilon,\overline{y}^\varepsilon)\in
\Omega_{N_{\varepsilon}}$.
From definition of $\alpha_\varepsilon$ and  Lemma \ref{lemma_phi2=0}, we have
\begin{eqnarray*}
  \varphi_1^{N_\varepsilon} (\overline{x}^\varepsilon,\overline{y}^\varepsilon)&=&
  \frac{1}{n_\varepsilon}\left[\sum_{k=1}^{N_\varepsilon} s_k\Big((1-\alpha_\varepsilon)
  |x'_k|^2-|y'_k|^2\Big)\right] \\
   &=& \frac{1}{n_\varepsilon}\left[\sum_{k=1}^{N_\varepsilon}
   s_k\Big(|x_k|^2-|y_k|^2\Big)-\alpha_\varepsilon \sum_{k=1}^{N_\varepsilon}
   s_k  |x_k|^2\right] \\
   &=& 0.
\end{eqnarray*}
Also, $\varphi_2^{N_\varepsilon}(\overline{x}^\varepsilon,\overline{y}^\varepsilon)=
\frac{\sqrt{1-\alpha_\varepsilon}}{n_\varepsilon^2}\varphi_2^{N_\varepsilon}(x',y')=0$ and $\varphi_3^{N_\varepsilon}(\overline{x}^\varepsilon,\overline{y}^\varepsilon)=0$, since $(\overline{x}^\varepsilon,\overline{y}^\varepsilon)\in
\bS_{\C^{2N_\varepsilon}}.$
Therefore, $(\overline{x}^\varepsilon,\overline{y}^\varepsilon)\in
\Omega_{N_{\varepsilon}}$.

Now, we will define a vector based on  $(\overline{x}^\varepsilon,\overline{y}^\varepsilon)$, which is in $\Omega$,  and prove that this vector converges to $(x_{\parallel},y_{\parallel})$ as $\varepsilon \rightarrow 0$.

 Let $\hat{x}^\varepsilon\in \Hi$ be given by $\hat{x}^\varepsilon_k=\overline{x}^\varepsilon_k$, for $1\leq k\leq N_\varepsilon$ and $0$ for $k>N_\varepsilon$, and let $\hat{y}^\varepsilon$ be defined in the same way.
  Note that, since $(\overline{x}^\varepsilon,\overline{y}^\varepsilon)\in
\Omega_{N_{\varepsilon}}$, $(\hat{x}^\varepsilon,\hat{y}^\varepsilon)\in \Omega$.
 We have, from (\ref{x_barra_epsilon}), that
\begin{eqnarray*}
\|\hat{x}^\varepsilon_\parallel- x_\parallel\|_{\ell^1} &=& \sum_{k=1}^\infty\left|\hat{x}^\varepsilon_{\parallel,k}-
  {x}_{\parallel,k}\right|\\
&=&
 \sum_{k=1}^{N_\varepsilon}\left||
 {\overline{x}}_{k}^\varepsilon|^2-
  |{x}_{k}|^2\right| +
  \sum_{k=N_\varepsilon+1}^\infty
  |{x}_{k}|^2\\
&=&
 \sum_{k=1}^{N_\varepsilon}\left|\frac{1-\alpha_\varepsilon}{n^2_\varepsilon}|
 {x}_{k}|^2-
  |{x}_{k}|^2\right| +
  \sum_{k=N_\varepsilon+1}^\infty
  |{x}_{k}|^2\\
&\leq &
 \dfrac{|1-\alpha_\varepsilon-n_\varepsilon^2|}{n_\varepsilon^2}\|x\|^2_{\ell^2}
 + \sum_{k=N_\varepsilon+1}^\infty
|{x}_{k}|^2.
\end{eqnarray*}
Since $\alpha_\varepsilon\rightarrow 0,
n_\varepsilon\rightarrow 1$  and $\sum_{k=N_\varepsilon+1}^\infty
  |{x}_{k}|^2\rightarrow 0$ as $\varepsilon \rightarrow 0$, it follows that $\|\hat{x}^\varepsilon_\parallel- x_\parallel\|_{\ell^1}\longrightarrow 0$.
Analogously, we can show that $\|\hat{y}^\varepsilon_\parallel- y_\parallel\|_{\ell^1}\longrightarrow 0$ and (\ref{convergence}) follows.

\end{proof}

\bigskip

\textbf{Case 2.} $0<s_1|x_1||y_1|=\sum_{k=2}^\infty s_k|x_k||y_k|$

Part of the proof of case 2 is analogous to case 1, so we only give details of the new arguments.

\begin{lemma}\label{c2lemma_phi2=0}
  Let $(x,y)\in \Omega$.
   If $s_1|x_1y_1|=\sum_{k=2}^\infty s_k|x_k||y_k|$, for every $\varepsilon>0$, there
  is $M\in\N$ such that for $N>M$, there is $(x',y')\in\C^{2N}$ and $0\leq t'<\varepsilon$ such that
  \begin{eqnarray*}
   {x}_1'= \sqrt{|x_1|^2-t'},& &\quad {y'}_1 = \sqrt{|y_1|^2-t'}, \\
 \quad \;\;\; {x}_k'=-|x_k|,\quad \;\;\quad & &\; \quad y_k'= |y_k|,\quad   k=2,\dots, N,
  \end{eqnarray*}
  and $\varphi_2^N(x',y')=0$.
  \end{lemma}

 \begin{proof}
 Without loss of generality, assume that $|x_1|\geq|y_1|$. Consider the function
 \[
  \xi: [0,|y_1|^2]\longrightarrow [0, s_1|x_1||y_1|]
 \]
 \[
 \xi(t)=s_1\sqrt{|x_1|^2-t}\sqrt{|y_1|^2-t}.
 \]
  Note that $\xi$ is continuous, strictly decreasing and onto.
  Therefore, it has a continuous inverse $\xi^{-1}$.
From hypothesis, we have:
 \begin{equation}\label{xi-1}
   \xi^{-1}(s_1|x_1||y_1|)=\xi^{-1}\left(\sum_{k=2}^\infty s_k|x_k||y_k|\right)=0.
 \end{equation}

By continuity of $\xi^{-1}$, for all $\varepsilon>0$, there exists $\delta>0$ such that
 \[
  \left|\xi^{-1}\left(\sum_{k=2}^{\infty} s_k|x_k||y_k|\right)-
  \xi^{-1}\left(\sum_{k=2}^{N} s_k|x_k||y_k|\right)\right|<\varepsilon,
  \]
  whenever
  \[
  \left|\sum_{k=2}^{\infty} s_k|x_k||y_k|-
\sum_{k=2}^{N} s_k|x_k||y_k|\right|<\delta.
  \]
 From (\ref{xi-1}), it follows that
 \begin{equation}\label{xi-1epsilon}
   \left|\xi^{-1}\left(\sum_{k=2}^{N} s_k|x_k||y_k|\right)\right|<\varepsilon, \text{ whenever } \sum_{k=N+1}^\infty
s_k|x_k||y_k| <\delta.
 \end{equation}

For $N \in \N$ satisfying $\sum_{k=N+1}^\infty
s_k|x_k||y_k| <\delta $, there is a $t'<\varepsilon$, in particular $t'=\xi^{-1}\left(\sum_{k=2}^{N} s_k|x_k||y_k|\right)$, such that
\begin{equation}\label{xi_t'}
  \xi(t')=\sum_{k=2}^{N} s_k|x_k||y_k|=s_1\sqrt{|x_1|^2-t'}\sqrt{|y_1|^2-t'}.
\end{equation}
 Define $(x',y')\in\R^{2N}$ as
 \begin{eqnarray*}
   {x}_1' &=& \sqrt{|x_1|^2-t'},\quad {y'}_1 = \sqrt{|y_1|^2-t'}, \\
    {x}_k' &=& -|x_k|, \quad  y_k'= |y_k|,\; 2\leq k\leq N.
 \end{eqnarray*}
 Therefore, from (\ref{xi_t'}) we have
 \[
 \varphi_2^{N}(x',y')=\sum_{k=1}^N s_kx_k'^*y_k'=s_1\sqrt{|x_1|^2-t'}\sqrt{|y_1|^2-t'}-\sum_{k=2}^{N} s_k|x_k||y_k|=0.
\]
  \end{proof}

  \begin{lemma}\label{c2lemma_omega_nonempty}
Let $(x,y)\in \Omega$. Then there is $(\hat{x}^\varepsilon, \hat{y}^\varepsilon)\in \Omega$ with finite support
such that
\begin{equation}\label{c2convergence}
(\hat{x}^\varepsilon_\parallel, \hat{y}^\varepsilon_\parallel) \longrightarrow_{\ell_1} (x_\parallel,y_\parallel), \quad
\text{as}\quad \varepsilon\longrightarrow 0.
\end{equation}
\end{lemma}

\begin{proof}
Fix $\varepsilon>0$ and choose $N>\frac{1}{\varepsilon}$ such that
\begin{equation}\label{c2series pequenas}
\sum_{k=N+1}^{\infty} s_k|x_k|^2<\frac{\varepsilon}{2} \quad\text{and}\quad
 \sum_{k=N+1}^{\infty} s_k|y_k|^2<\frac{\varepsilon}{2}.
\end{equation}
Pick $M$ from Lemma \ref{c2lemma_phi2=0} and let $N_\varepsilon>\max\{N,M\}$. Then there is $(x',y')\in\C^{2N_\varepsilon}$ such that  $\varphi_2^{N_{\varepsilon}}(x',y')=0$. We will now proceed as in Lemma \ref{lemma_omega_nonempty}. All the steps are very similar to those in the lemma, for that reason we will only give an overall explanation. We start by creating a new vector
\begin{align*}\label{c2x_barra_epsilon}
  &(\overline{x}^\varepsilon,\overline{y}^\varepsilon)=
  \frac{(\sqrt{1-\alpha_\varepsilon}x',y')}{n_\varepsilon}
\end{align*}
with $ n_\varepsilon=\|(\sqrt{1-\alpha_\varepsilon} x',y')\|$  and
  $\alpha_\varepsilon=\dfrac{\sum_{k=1}^{N_\varepsilon} s_k\Big(|x_k'|^2-|y_k'|^2\Big)}
 {\sum_{k=1}^{N_\varepsilon} s_k|x_k'|^2}.$

We can prove, as before, that $\alpha_\varepsilon\rightarrow 0$ as
$\varepsilon\rightarrow 0$ and that $(\overline{x}^\varepsilon,\overline{y}^\varepsilon)\in
\Omega_{N_{\varepsilon}}$, (that is $\varphi^{N_\varepsilon}(\overline{x}^\varepsilon,\overline{y}^\varepsilon)=0$). We then define a new vector $(\hat{x}^\varepsilon, \hat{x}^\varepsilon )\in \Hi$, with $\hat{x}^\varepsilon_k=\overline{x}^\varepsilon_k$, for $1\leq k\leq N_\varepsilon$ and $0$ otherwise, $\hat{y}^\varepsilon$ is defined in the same way. We then prove that $\|\hat{x}^\varepsilon_\parallel- x_\parallel\|_{\ell^1}\longrightarrow 0$ and $\|\hat{y}^\varepsilon_\parallel- y_\parallel\|_{\ell^1}\longrightarrow 0$ as $\varepsilon \longrightarrow 0$.
\end{proof}

In both case 1 and case 2 we proved the existence of a vector with finite support in
$\Omega$ that converges to $(x_{\parallel}, y_{\parallel})$
(see Lemma \ref{lemma_omega_nonempty} and Lemma \ref{c2lemma_omega_nonempty}).
So, now we are able to characterize $\underline{v}$ and $\overline{v}$ in Theorem
\ref{theo_upperbild_spectrum}, for a normal operator $T\in\B(\Hi)$.

For each pair $(k,j)$ with $k\neq j$, we introduce the real numbers $\underline c_{kj}=\min [d_k,d_j^*]\cap \bR$ and $\overline c_{kj}=\max [d_k,d_j^*]\cap \bR$, where
$d_k=h_k+s_ki$ are the eigenvalues of $T$, that is,
$\sigma_{r}(T)=\bigcup_k d_k$. Note that if $s_k>0$ or $s_j>0$ then $\underline c_{kj}=\overline c_{kj}$, however in the case where $s_k=s_j=0$, we have $\underline c_{kj}=\min\{h_k,h_j\}\leq \max\{h_k,h_j\}=\overline c_{kj}$. Furthermore, define
\[
\underline{c}=\inf\{\underline c_{kj}: \, k,j\in \N,\, k\neq j\} \quad \text{and} \quad
 \overline{c}=\sup\{\overline c_{kj}: \, k,j\in \N,\, k\neq j\}.
\]

\begin{theorem}\label{theo_vmin_vmax}
Let $T\in\B(\Hi)$ be a normal operator such that $\sigma_{r}(T)=\bigcup_k d_k$, with $d_k=h_k+s_ki$,
$s_k\geq 0$.
Then
\[
\underline{v}=\underline{c} \quad \text{and} \quad
 \overline{v}=\overline{c}.
\]
\end{theorem}

\begin{proof}
In the proof we only deal with $\underline{v}=\underline{c}$. The other equality follows the same reasoning.
From Proposition \ref{prop normal app diagonal} it is enough to consider a diagonal complex
operator $De_k=e_kd_k$ with $d_k=h_k+s_ki$, $s_k\geq 0$. We wish to prove that for any $(x,y) \in\Omega$, $f(x,y) \geq \underline c$. For a given $(x,y) \in\Omega$, let $\mathcal{K}'=\{j \in \N: |x_j|^2+|y_j|^2>0 \}$. We will start by considering the case where $s_jx_jy_j=0$ for any $j \in \mathcal{K}'$.

Let $S_0=\{j \in \N: s_j=0\}$. If $j\in S_0$ then, for any $k \in \N$, $\underline{c}_{jk} = h_j$ (if $s_k\neq0$) or $\underline{c}_{jk}= \min\{h_j,h_k\}$ (if $s_k=0$), clearly $h_j\geq \underline{c}_{jk}\geq \underline c$. Thus
\begin{align*}
f(x,y)=&\sum_j h_j(|x_j|^2+|y_j|^2)\\
=&\sum_{j \in S_0}h_j(|x_j|^2+|y_j|^2)+ \sum_{j\in \K} h_j(|x_j|^2+|y_j|^2)\\
\geq &\cmin\sum_{j \in S_0}(|x_j|^2+|y_j|^2)+ \sum_{j\in \K} h_j(|x_j|^2+|y_j|^2).
\end{align*}
where, for the given $(x,y) \in\Omega$,  we let  $\mathcal{K}=\{j \in \N: |x_j|^2+|y_j|^2>0, s_j>0 \}$. Then to prove that $f(x,y)\geq \cmin$ we just need to see that $ \sum_{j\in \K} h_j(|x_j|^2+|y_j|^2)\geq \cmin  \sum_{j\in \K} (|x_j|^2+|y_j|^2)$. If $\K=\emptyset$ we are done, so we assume that $\K \neq \emptyset$. Since $s_jx_jy_j=0$ and $j \in \K$, we either have $x_j \neq 0$ and $y_j =0$, or $x_j=0$ and $y_j \neq 0$. From condition (I), it follows that there are $k,j\in\K$ with $k\neq j$. For such pairs, from the definition it results that
\[
\underline c_{kj}=h_k\frac{s_j}{s_j+s_k}+h_j\frac{s_k}{s_j+s_k}.
\]
Then if $s_k a^2=s_j b^2$, for reals $a$ and $b$, we have that $\frac{a^2}{a^2+b^2}=\frac{s_j}{s_j+s_k}$ and therefore
\begin{equation}\label{ineq}
\quad h_k a^2+ h_j b^2=\Big(a^2+ b^2\Big)\underline{c}_{kj}\geq \Big( a^2+b^2\Big)\underline{c},
\end{equation}
the inequality being trivially verified if $a=b=0$.

Furthermore, if $s_k a^2=\sum_{ j \in \K\setminus \{k\}} s_j b_j^2$, for reals $a$ and $b_j$ with $k \in \K$,  partioning $a^2=\sum_{ j \in \K\setminus \{k\}}a_j^2$ with $s_k a_j^2=s_j b_j^2$, we conclude, using (\ref{ineq}), that
\begin{equation}\label{ineq2}
h_k a^2+ \sum_{ j \in \K\setminus \{k\}} h_j b_j^2 = \sum_{ j \in \K\setminus \{k\}} \big(h_k a_j^2+  h_j b_j^2\big) \geq\Big( a^2+\sum_{ j \in \K\setminus \{k\}}  b_j^2\Big)\underline{c}.
\end{equation}
%
Now, from condition (I), since $\sum_{k\in \N}  s_k|x_k|^2=\sum_{k\in \K}  s_k|x_k|^2=\sum_{j\in \K} s_j|y_j|^2$,
 it is a tedious computation to verify that the vectors $y^{(k)} \in \mathcal{H}_{\C}$ given by
\[
y^{(k)}=y \Bigg(\frac{s_k|x_k|^2}{\sum_{k \in \K} s_k|x_k|^2}\Bigg)^{\frac{1}{2}}
\]
satisfy for each coordinate $ j \in \K$, $|y_j|^2= \sum_{ j \in \K\setminus \{k\}}|y^{(k)}_j|^2$ and, for each $k \in \K$, $s_k|x_k|^2= \sum_{ j \in \K\setminus \{k\}} s_j|y^{(k)}_j|^2$.

 Then,
\begin{align*}
\sum_{k\in \K} h_k|x_k|^2+& \sum_{j\in \K} h_j |y_j|^2= \sum_{k\in \K} h_k|x_k|^2+ \sum_{j\in \K}h_j \Big(\sum_{k\in \K\setminus \{j\}} |y^{(k)}_j|^2\Big)\\
=&\sum_{k\in \K} \Big(h_k|x_k|^2+ \sum_{j\in \K\setminus \{k\}} h_j |y^{(k)}_j|^2\Big)
\geq  \sum_{k\in \K} \Big( |x_k|^2+\sum_{j\in \K\setminus \{k\}}  |y^{(k)}_j|^2\Big)\underline{c},
\end{align*}
 where the last inequality follows from (\ref{ineq2}). Since $\sum_{k\in \K}  \sum_{j\in \K\setminus \{k\}}  |y^{(k)}_j|^2=\sum_{j\in \K} \sum_{k\in \K\setminus \{j\}}|y^{(k)}_j|^2=\sum_{j\in \K}  |y_j|^2$, we have the desired inequality.
We are left with the case when there is $j \in\mathcal{K}'$ such that $s_jx_jy_j\neq 0$.
Let $(x,y)\in \Omega$ and $\varepsilon>0$. From Lemma \ref{lemma_omega_nonempty} and Lemma \ref{c2lemma_omega_nonempty}, there is $(\hat{x}^\varepsilon, \hat{y}^\varepsilon)\in\Omega$
such that
\[
(\hat{x}^\varepsilon_{||}, \hat{y}^\varepsilon_{||}) \longrightarrow_{\ell^1} (x_\parallel,y_\parallel), \quad
\text{as}\quad \varepsilon\longrightarrow 0.
\]
By $\ell^1$ continuity of $f$ and $f(z,w)=f(z_{||}, w_{||})$ it follows that
\[
f(\hat{x}^\varepsilon, \hat{y}^\varepsilon)=f(\hat{x}^\varepsilon_{||}, \hat{y}^\varepsilon_{||})\longrightarrow f(x_\parallel,y_\parallel)=f(x,y).
\]

For $z \in \C^\N $ and $N \in \N$ let $\pi^N(z) \in \C^N$ be the projection over the first $N$ coordinates. Since $\hat{x}^\varepsilon_k=\hat{y}^\varepsilon_k=0$, for $k > N_\varepsilon$, then
\[
f^{N_\varepsilon}\Big( \pi\big(\hat{x}^\varepsilon, \hat{y}^\varepsilon\big) \Big) = f\big(\hat{x}^\varepsilon, \hat{y}^\varepsilon\big)
\]
where $ \pi\big(\hat{x}^\varepsilon, \hat{y}^\varepsilon\big)=\big(\pi^{N_\varepsilon}( \hat{x}^\varepsilon), \pi^{N_\varepsilon}(\hat{y}^\varepsilon) \big) $. We know from Lemma \ref{lemma_omega_nonempty} and Lemma \ref{c2lemma_omega_nonempty} that $\pi\big(\hat{x}^\varepsilon, \hat{y}^\varepsilon\big) \in \Omega^{N_\varepsilon}$. From \cite{CDM4} we know that
\[
f^{N_\varepsilon}\big(\pi\big(\hat{x}^\varepsilon, \hat{y}^\varepsilon\big)\big)\geq \vmin_{N_\varepsilon}=\min \{ c_{kj}: k \neq j, k,j \leq N_\varepsilon\}.
\]
We can then conclude, since $\vmin_{N_\varepsilon}\geq \inf \{ c_{kj}: k \neq j\}$, that
\[
f\big(\hat{x}^\varepsilon, \hat{y}^\varepsilon\big) \geq \inf \{ c_{kj}: k \neq j\}.
 \]
We conclude, taking limits, that $f(x,y) \geq \inf \{ c_{kj}: k \neq j\}$. Since this is true for any $(x,y) \in \Omega$, $\vmin \geq \inf \{ c_{kj}: k \neq j\}$.




Conversely, given $c_{kj}$ let $x=\sqrt{\alpha_{kj}}e_{k}$ and
$y=\sqrt{1-\alpha_{kj}}e_{j}$. One easily see that
$f(x,y)=\alpha_{kj}h_k+(1-\alpha_{kj})h_j=c_{kj}$.
Hence, $\underline{v}\leq \inf\{c_{kj}: k,j\in \N,\; k\neq j\}$
\end{proof}

\section*{Acknowledgement}
The authors would like to thank the anonymous referees for their suggestions that helped to improve the paper.

\end{document}